\documentclass[12pt]{article}

\usepackage{amsmath,amsthm,amssymb,graphicx,pb-diagram,geometry,mathrsfs}
\usepackage{graphicx,epstopdf}
\usepackage{amsmath,latexsym,amssymb,amsthm}
\usepackage{sectsty}
\usepackage[lofdepth,lotdepth]{subfig}
\usepackage{array}
\usepackage{hyperref}
\usepackage{mathtools}
\usepackage{blkarray, bigstrut}
\usepackage{enumerate}   
\usepackage{adjustbox}

\theoremstyle{plain}
\newtheorem{theorem}{Theorem}
\newtheorem{lemma}[theorem]{Lemma}
\newtheorem{corollary}[theorem]{Corollary}

\theoremstyle{definition}
\newtheorem{definition}[theorem]{Definition}
\newtheorem{example}[theorem]{Example}

\theoremstyle{remark}

\begin{document}

\title{A characterization of $Q$-polynomial distance-regular graphs using the intersection numbers}
\author{Supalak  Sumalroj\\
	\small Department of Mathematics, Silpakorn University, 
	Nakhon Pathom, Thailand\\
	\small sumalroj\_s@silpakorn.edu
}
\date{}
\maketitle

\begin{abstract}
	We consider a primitive distance-regular graph $\Gamma$ with diameter at least $3$.
	We use the intersection numbers of $\Gamma$ to find a positive semidefinite matrix $G$ with integer entries.
	We show that $G$ has determinant zero if and only if $\Gamma$ is $Q$-polynomial.
\end{abstract}


\section{Introduction}

Let $\Gamma$ denote a distance-regular graph with diameter $D\geq3$. 
In the literature there are a number of characterizations for the $Q$-polynomial condition on $\Gamma$.
There is the balanced set characterization \cite[Theorem 1.1]{T_characterization of P and Q}, \cite[Theorem 3.3]{T_New}.
There is a characterization involving the dual distance matrices \cite[Theorem 3.3]{T_New}.
There is a characterization involving the intersection numbers $a_i$ \cite[Theorem 3.8]{Pascasio}; cf. \cite[Theorem 5.1]{Hanson}.
There is a characterization involving a tail in a representation diagram \cite[Theorem 1.1]{Jurisic_Terwilliger_Zitnik}.
There is a characterization involving a pair of primitive idempotents \cite[Theorem 1.1]{Kurihara and Nozaki}; cf. \cite[Theorem 1.1]{Nomura and Terwilliger}.

In this paper we obtain the following characterization of the $Q$-polynomial property.
Assume $\Gamma$ is primitive.
We use the intersection numbers of $\Gamma$ to find a positive semidefinite matrix $G$ with integer entries.
We show that $G$ has determinant zero if and only if $\Gamma$ is $Q$-polynomial.
Our main result is Theorem \ref{primitive}.

\section{Preliminaries}
	
Let $X$ denote a nonempty finite set. 
Let ${{\rm{Mat}}}_X(\mathbb{C})$ denote the $\mathbb{C}$-algebra consisting of the matrices whose rows and columns are indexed by $X$ and whose entries are in $\mathbb{C}$. 
For $B\in{{\rm{Mat}}_X(\mathbb{C})}$ let $\overline{B}$ and $B^t$ denote the complex conjugate and the transpose of $B$, respectively.
Let $V=\mathbb{C}^X$ denote the vector space over $\mathbb{C}$ consisting of column vectors with coordinates indexed by $X$ and entries in $\mathbb{C}$.
Observe that ${{\rm{Mat}}}_X(\mathbb{C})$ acts on $V$ by left multiplication. 
We endow $V$ with the Hermitean inner product $(\;,\;)$ such that  $(u,v)=u^t\overline{v}$ for all $u,v\in{V}$.
The inner product $(\;,\;)$ is positive definite.
For $B\in{{\rm{Mat}}_X(\mathbb{C})}$ we obtain $(u,Bv)=(\overline{B}^tu,v)$ for all $u,v\in{V}$.
We endow ${\rm{Mat}}_X(\mathbb{C})$ with the Hermitean inner product $\langle \;,\;\rangle$  such that  $\langle R,S\rangle =tr(R^t\overline{S})$ for all $R, S\in{{\rm{Mat}}_X(\mathbb{C})}$.
The inner product $\langle \;,\;\rangle$ is positive definite.

Let $\Gamma=(X,R)$ denote a finite, undirected, connected graph, without loops or multiple edges, with vertex set $X$ and edge set $R$.
Let $\partial$ denote the shortest path-length distance function for $\Gamma$.
Define the diameter $D:=\text{max}\{\partial(x,y)| x, y\in{X}\}$. For a vertex $x\in{X}$ and an integer $i\geq0$ define
$\Gamma_i(x)=\{y\in{X}|\partial(x,y)=i\}$.
For notational convenience abbreviate $\Gamma(x)=\Gamma_1(x)$.
For an integer $k\geq 0$, we call the graph $\Gamma$ \textit{regular with valency $k$} whenever $|\Gamma(x)|=k$ for all $x\in{X}$.
We say that $\Gamma$ is \textit{distance-regular} whenever for all integers $h, i, j$ $(0\leq h, i, j\leq D)$ and for all $x, y \in{X}$ with $\partial(x, y) = h$, the number
$$p^h_{ij}:=|\Gamma_i(x)\cap\Gamma_j(y)|$$
is independent of $x$ and $y$.
The integers $p^h_{ij}$ are called the \textit{intersection numbers} of $\Gamma$.
From now on we assume $\Gamma$ is distance-regular with diameter $D\geq 3$.
We abbreviate $c_i:=p^i_{1,i-1}$ $(1\leq i \leq D)$, $a_i:=p^i_{1i}$ $(0\leq i \leq D)$, $b_i:=p^i_{1,i+1}$ $(0\leq i \leq D-1)$, $k_i:=p^0_{ii}$ $(0\leq i \leq D)$, and $c_0=0$, $b_D=0$.
Observe that $\Gamma$ is regular with valency $k=b_0$ and $c_i + a_i + b_i = k$ $(0\leq i \leq D)$.
By \cite[p. 127]{BCN} the following holds for $0\leq h, i, j\leq D$:
(i) $p^h_{ij}=0$ if one of $h, i, j$ is greater than the sum of the other two; and (ii) $p^h_{ij}\neq0$ if one of $h, i, j$ equals the sum of the other two.
For $0\leq i \leq D$, let $A_i$ denote the matrix in ${\rm{Mat}}_X(\mathbb{C})$ with $(x,y)$-entry
\[ (A_i)_{xy}=
\begin{cases}
1       & \quad \text{if } \quad \partial(x,y)=i,\\
0  & \quad \text{if } \quad \partial(x,y)\neq i,\\
\end{cases}
\qquad \qquad x,y\in{X}.
\]
We call $A_i$ the \textit{$i$-th distance matrix} of $\Gamma$.
We call $A=A_1$ the \textit{adjacency matrix}  of $\Gamma$. 
Observe that $A_i$ is real and symmetric for $0\leq i \leq D$. 
Note that $A_0 = I$, where $I$ is the identity matrix.
Observe that $\sum_{i=0}^{D} A_i=J$, where $J$ is the all-ones matrix in ${\rm{Mat}}_X(\mathbb{C})$.
Observe that for $0 \leq i, j \leq D$, $A_iA_j =\sum_{h=0}^{D} p^h_{ij}A_h$.
For integers $h,i,j$ $(0\leq h,i,j \leq D)$ we have 
\begin{align} 
&p^h_{0j}=\delta_{hj} \label{p^h_{0j}} \\
&p^0_{ij}=\delta_{ij}k_i \label{p^0_{ij}}  \\
&p^h_{ij}=p^h_{ji} \label{p^h_{ij}=p^h_{ji}} \\
&k_hp^h_{ij}  = k_ip^i_{hj} = k_jp^j_{ih}.  \label{k_hp^h_{ij}}
\end{align}

\noindent For $0 \leq i,j\leq D$ we have $A(A_i A_j)= (A A_i)A_j$. 
This gives a recursion
\begin{equation}\label{formula ai,bi,ci}
c_{i+1}p^h_{i+1,j}+a_{i}p^h_{ij}+b_{i-1}p^h_{i-1,j}=
c_{h}p^{h-1}_{ij}+a_{h}p^h_{ij}+b_{h}p^{h+1}_{ij} 
\end{equation}
\noindent that can be used to compute the intersection numbers.

Let $M$ denote the subalgebra of ${\rm{Mat}}_X(\mathbb{C})$ generated by $A$. 
By \cite[p. 127]{BCN} the matrices $A_0,A_1,...,A_D$ form a basis for $M$.
We call $M$ the \textit{Bose-Mesner algebra of} $\Gamma$. 
By \cite[p. 45]{BCN}, $M$ has a basis $E_0,E_1,...,E_D$ such that 
(i) $E_0 = |X|^{-1}J$; 
(ii) $\sum_{i=0}^{D} E_i=I$;
(iii) $E_i^t=E_i$ $(0\leq i \leq D)$;
(iv) $\overline{E_i}=E_i$ $(0\leq i \leq D)$;
(v) $E_iE_j = \delta_{ij}E_i$ $(0\leq i,j \leq D)$.
The matrices $E_0, E_1,...,E_D$ are called the \textit{primitive idempotents} of $\Gamma$, and $E_0$ is called the \textit{trivial} idempotent.
%
%
For $0\leq i \leq D$ let $m_i$ denote the rank of $E_i$.
%
%
%
\noindent Let $\lambda$ denote an indeterminate.
Define polynomials $\{v_i\}_{i=0}^{D+1}$ in $\mathbb{C}[\lambda]$ by $v_0=1$, $v_1=\lambda$, and
\begin{align*}
\lambda v_i=c_{i+1}v_{i+1}+a_iv_i+b_{i-1}v_{i-1} \qquad  (1\leq i\leq D),
\end{align*}

\noindent where $c_{D+1}:=1$.
By \cite[p. 128]{BCN} and \cite[Lemma 3.8]{T_subconstituentI}, the following hold:
(i) $deg$ $v_i=i$ $(0\leq i \leq D+1)$; 
(ii) the coefficient of $\lambda^i$ in $v_i$ is $(c_1c_2\dots c_i)^{-1}$ $(0\leq i \leq D+1)$;
(iii) $v_i(A)=A_i$ $(0\leq i \leq D)$;
(iv) $v_{D+1}(A)=0$;
(v) the distinct eigenvalues of $\Gamma$ are precisely the zeros of $v_{D+1}$; call these $\theta_0, \theta_1, ..., \theta_D$.
Define a matrix $B\in{{\rm{Mat}}_{D+1}(\mathbb{C})}$ as follows:

\[
B=\begin{bmatrix}
a_0    & b_0       &          & 		& \bf{0}  	\\
c_1    & a_1       &  b_1     &   		& 				\\
	   & c_2       &  a_2     &  \ddots & 				\\
	   &		   &  \ddots  & \ddots  & b_{D-1}		\\
\bf{0}  &      &          & c_D 	& a_D
\end{bmatrix}.
\]
We call $B$ the \textit{intersection matrix} of $\Gamma$.
Note that $A$ has the same minimal polynomial as $B$.
Moreover the minimal polynomial of $B$ is the characteristic polynomial of $B$.
For an eigenvalue $\theta$ of $B$ we have $vB=\theta v$ where $v$ is a row vector $v=(v_0(\theta),v_1(\theta),...,v_D(\theta))$.
Define polynomials $\{u_i\}_{i=0}^{D}$ in $\mathbb{C}[\lambda]$ by 
$u_0=1$, $u_1=\lambda/k$, and
\begin{align*}
\lambda u_i=c_{i}u_{i-1}+a_iu_i+b_{i}u_{i+1} \qquad (1\leq i\leq D-1).
\end{align*}

\noindent Observe that $u_i=v_i/k_i$ $(0\leq i \leq D)$.
For an eigenvalue $\theta$ of $B$ we have $Bu=\theta u$ where $u$ is a column vector $u=(u_0(\theta),u_1(\theta),...,u_D(\theta))^t$.
By \cite[p. 131, 132]{BCN}, 
\begin{align} 
&A_j=\displaystyle\sum_{i=0}^{D}v_j(\theta_i)E_i  &(0\leq j \leq D), \label{A_j}\\ 
&E_j=|X|^{-1}m_j\displaystyle\sum_{i=0}^{D} u_i(\theta_j)A_i  &(0\leq j \leq D). \label{E_j}
\end{align}

\noindent Since $E_iE_j = \delta_{ij}E_i$ and by (\ref{A_j}), (\ref{E_j}) we have $A_jE_i = E_iA_j$ $(0 \leq i,j \leq D)$.

For $1\leq i \leq D$ let $\Gamma_i$ denote the graph with vertex set $X$ where vertices are adjacent in $\Gamma_i$ whenever they are at distance $i$ in $\Gamma$.
We observe that $\Gamma_1=\Gamma$.
The graph $\Gamma$ is said to be \textit{primitive} whenever $\Gamma_i$ is  connected for $1\leq i \leq D$.

\begin{lemma} $\rm{(See ~\cite [Proposition~4.4.7]{BCN}.)}$ \label{primitive_S_alt}
	Assume $\Gamma$ is primitive.
	Then $u_i(\theta_j)\neq 1$ for $1\leq i,j \leq D$.
\end{lemma}

We now define a  matrix $S\in{{\rm{Mat}}_{D+1}(\mathbb{C})}$.

\begin{definition} \label{matrixS}
	Let $S\in{{\rm{Mat}}_{D+1}(\mathbb{C})}$ denote the transition matrix from the basis $\{A_i\}_{i=0}^D$ of $M$ to the basis $\{E_i\}_{i=0}^D$ of $M$.
	Thus 	
	\begin{align*} 
	& & &\qquad E_j=\displaystyle\sum_{i=0}^{D} S_{ij}A_i  &(0\leq j \leq D),\\
	& & &\qquad A_j=\displaystyle\sum_{i=0}^{D} (S^{-1})_{ij}E_i  &(0\leq j \leq D).
	\end{align*}
\end{definition}

\begin{lemma} \label{entries of S}
	The entries of $S$ and $S^{-1}$ are given below.
	For $0\leq i,j \leq D$,
	\begin{align*}
	S_{ij}=|X|^{-1}m_ju_i(\theta_j), \qquad \qquad
	(S^{-1})_{ij}=v_j(\theta_i).
	\end{align*}
\end{lemma}
\begin{proof}
	Immediate from Definition \ref{matrixS} and (\ref{A_j}), (\ref{E_j}).
\end{proof}



We recall the $Q$-polynomial property.
Let $\circ$ denote the entry-wise multiplication in ${\rm{Mat}}_X(\mathbb{C})$.
Note that $A_i\circ A_j = \delta_{ij}A_i$ for $0 \leq i, j \leq D$.
So $M$ is closed under $\circ$. 
By \cite[p. 377]{T_subconstituentI}, there exist scalars $q^h_{ij}\in{\mathbb{C}}$ such that
\begin{equation}\label{E_i_circE_j} 
\begin{split}
E_i\circ E_j=|X|^{-1}\displaystyle\sum_{h=0}^{D} q^h_{ij}E_h  \quad \qquad (0\leq i,j \leq D).
\end{split}
\end{equation}

\noindent We call the $q^h_{ij}$ the \textit{Krein parameters} of $\Gamma$. 
By \cite[p. 48, 49]{BCN}, these parameters are real and nonnegative for $0\leq h,i,j\leq D$.
The graph $\Gamma$ is said to be \textit{$Q$-polynomial} with respect to the ordering $E_0,E_1,...,E_D$ whenever the following hold for $0\leq h,i,j\leq D$: 
(i) $q^h_{ij}=0$ if one of $h, i, j$ is greater than the sum of the other two; and (ii) $q^h_{ij}\neq0$ if one of $h, i, j$ equals the sum of the other two.
Let $E$ denote a primitive idempotent of $\Gamma$.
We say that $\Gamma$ is \textit{$Q$-polynomial with respect t}o $E$ whenever there exists a $Q$-polynomial ordering $E_0,E_1,...,E_D$ of the primitive idempotents such that $E=E_1$.

We recall the dual Bose-Mesner algebra of $\Gamma$.
Fix a vertex $x\in{X}$. 
For $0 \leq i \leq D$ let $E_i^*=E_i^*(x)$ denote the diagonal matrix in ${\rm{Mat}}_X(\mathbb{C})$ with $(y,y)$-entry
\[ (E_i^*)_{yy}=
\begin{cases}
1       & \quad \text{if } \quad \partial(x,y)=i,\\
0  & \quad \text{if } \quad \partial(x,y)\neq i,\\
\end{cases}
\qquad \qquad y\in{X}.
\]
We call $E_i^*$ the \textit{$i$-th dual idempotent} of $\Gamma$ with respect to $x$. 
Observe that (i) $\sum_{i=0}^{D} E_i^*=I$;
(ii)  $E_i^{*t}=E_i^*$ $(0\leq i \leq D)$;
(iii) $\overline{E_i^*}=E_i^*$ $(0\leq i \leq D)$; 
(iv) $E_i^*E_j^*=\delta_{ij}E_i^*$ $(0\leq i,j \leq D)$.
By construction $E_0^*,E_1^*,...,E_D^*$ are linearly independent.
Let $M^*=M^*(x)$ denote the subalgebra of ${\rm{Mat}}_X(\mathbb{C})$ with  basis $E_0^*,E_1^*,...,E_D^*$.
We call $M^*$ the \textit{dual Bose-Mesner algebra} of $\Gamma$ with respect
to $x$.

We now recall the dual distance matrices of $\Gamma$.
For $0 \leq i \leq D$ let $A_i^*=A_i^*(x)$ denote the diagonal matrix in ${\rm{Mat}}_X(\mathbb{C})$ with $(y,y)$-entry
\begin{equation}\label{A_i^*} 
\begin{split}
(A_i^*)_{yy}=|X|(E_i)_{xy} \qquad \qquad \qquad y\in{X}.
\end{split}
\end{equation} 
We call $A_i^*$ the \textit{dual distance matrix} of $\Gamma$ with respect to $x$ and $E_i$.
By \cite [p. 379]{T_subconstituentI}, the matrices $A_0^*,A_1^*,...,A_D^*$ 
form a basis for $M^*$.
Observe that
(i) $A_0^* = I$; 
(ii) $\sum_{i=0}^{D} A_i^*=|X|E_0^*$;
(iii) $A_i^{*t}=A_i^*$ $(0\leq i \leq D)$;
(iv) $\overline{A_i^*}=A_i^*$ $(0\leq i \leq D)$;
(v) $A_i^*A_j^* = \sum_{h=0}^{D} q^h_{ij}A_h^*$ $(0\leq i,j \leq D)$.
From (\ref{A_j}), (\ref{E_j}) we have
\begin{align}
&A_j^*=m_j\displaystyle\sum_{i=0}^{D}u_i(\theta_j)E_i^*  &(0\leq j \leq D), \label{A_j^*} \\
&E_j^*=|X|^{-1}\displaystyle\sum_{i=0}^{D} v_j(\theta_i)A_i^*  &(0\leq j \leq D). \label{E_j^*} 
\end{align}

\begin{lemma} \label{transS*}
	The matrix $|X|S$ is the transition matrix from the basis $\{E^*_i\}_{i=0}^D$ of $M^*$ to the basis $\{A^*_i\}_{i=0}^D$ of $M^*$.
	Thus
	\begin{align*} 
	&A_j^*=|X|\displaystyle\sum_{i=0}^{D} {S}_{ij}E_i^*  &(0\leq j \leq D),\\ 
	&E_j^*=|X|^{-1}\displaystyle\sum_{i=0}^{D} (S^{-1})_{ij}A_i^*  &(0\leq j \leq D). 
	\end{align*}
\end{lemma}
\begin{proof}
	Immediate from Lemma \ref{entries of S} and (\ref{A_j^*}), (\ref{E_j^*}).
\end{proof}




\section{The matrices $S^{alt}, (S^{-1})^{alt}, S'$}

Recall the matrix $S$ from Definition \ref{matrixS}.
We now modify the matrices $S,S^{-1}$ to obtain  $D\times D$ matrices $S^{alt},(S^{-1})^{alt}$ as follows:
\begin{align} 
&\qquad (S^{alt})_{ij}= S_{ij}-S_{0j}  &(1\leq i,j \leq D), \label{S_alt} \\
&\qquad (S^{-1})^{alt}_{ij}= (S^{-1})_{ij}  &(1\leq i,j \leq D). \label{S_alt_inverse}
\end{align}

\begin{lemma} \label{transS^alt}
	The following (i)--(iv) hold.
	\begin{enumerate}[(i)]
		\item $S^{alt}$ is the transition matrix from $\{A_2E_i^*A-AE_i^*A_2\}_{i=1}^D$ to $\{A_2A_i^*A-AA_i^*A_2\}_{i=1}^D$.
		\item $S^{alt}$ is the transition matrix from $\{A_3E_i^*-E_i^*A_3\}_{i=1}^D$ to $\{A_3A_i^*-A_i^*A_3\}_{i=1}^D$.
		\item $S^{alt}$ is the transition matrix from $\{A_2E_i^*-E_i^*A_2\}_{i=1}^D$ to $\{A_2A_i^*-A_i^*A_2\}_{i=1}^D$.
		\item $S^{alt}$ is the transition matrix from $\{AE_i^*-E_i^*A\}_{i=1}^D$ to $\{AA_i^*-A_i^*A\}_{i=1}^D$.
		\item $(S^{-1})^{alt}$ and $S^{alt}$ are inverses.
	\end{enumerate}
\end{lemma}
\begin{proof}
	$(i),(v)$ 
	For $1\leq j \leq D$ we write $A_2A_j^*A-AA_j^*A_2$ in terms of $\{A_2E_i^*A-AE_i^*A_2\}_{i=1}^{D}$.
	By Lemma \ref{transS*} and (\ref{S_alt}) and since $\sum_{i=0}^{D}E_i^*=I$, we have 
	\begin{align*}
	A_2A_j^*A-AA_j^*A_2
	&=|X|\displaystyle\sum_{i=0}^{D}(A_2E_i^*A-AE_i^*A_2)S_{ij}\\
	&=|X|(A_2E_0^*A-AE_0^*A_2)S_{0j}+|X|\displaystyle\sum_{i=1}^{D}(A_2E_i^*A-AE_i^*A_2)S_{ij} \\
	&=|X|(A_2(I-(E_1^*+\cdots+E_D^*))A-A(I-(E_1^*+\cdots+E_D^*))A_2)S_{0j}\\
	& \quad +|X|\displaystyle\sum_{i=1}^{D}(A_2E_i^*A-AE_i^*A_2)S_{ij} \\
	&=|X|\displaystyle\sum_{i=1}^{D}(A_2E_i^*A-AE_i^*A_2)(S_{ij}-S_{0j}) \\
	&=|X|\displaystyle\sum_{i=1}^{D}(A_2E_i^*A-AE_i^*A_2)(S^{alt})_{ij}.
	\end{align*}

	\noindent Next, for $1\leq j \leq D$ we write $A_2E_j^*A-AE_j^*A_2$ in terms of $\{A_2A_i^*A-AA_i^*A_2\}_{i=1}^{D}$.
	By Lemma \ref{transS*} and (\ref{S_alt_inverse}) and since $A_0^*=I$, we find 
	\begin{align*}
	A_2E_j^*A-AE_j^*A_2
	&=|X|^{-1}\displaystyle\sum_{i=0}^{D}(A_2A_i^*A-AA_i^*A_2)(S^{-1})_{ij}\\
	&=|X|^{-1}(A_2A_0^*A-AA_0^*A_2)(S^{-1})_{0j}\\
	& \quad +|X|^{-1}\displaystyle\sum_{i=1}^{D}(A_2A_i^*A-AA_i^*A_2)(S^{-1})_{ij} \\
	&=|X|^{-1}\displaystyle\sum_{i=1}^{D}(A_2A_i^*A-AA_i^*A_2)(S^{-1})_{ij} \\
	&=|X|^{-1}\displaystyle\sum_{i=1}^{D}(A_2A_i^*A-AA_i^*A_2)( S^{-1})^{alt}_{ij}.
	\end{align*}

	\noindent The result follows.
	
	\noindent $(ii)-(iv)$ Similar to the proof of $(i)$.
\end{proof}

Define a matrix 
\[
S'=\begin{bmatrix}
S^{alt}    		&        	&        	& \bf{0}  \\
& S^{alt}   &        	&   \\
&        	& S^{alt}  	&   \\
\bf{0}  	&        	&        	& S^{alt}  
\end{bmatrix},
\]
\noindent where $S^{alt}$ is from (\ref{S_alt}).
Observe that $S'$ is $4D\times 4D$.

\begin{lemma} \label{det_S'}
	$det(S')=\big(det(S^{alt})\big)^4$.
	Moreover $S'$ is invertible.
\end{lemma}
\begin{proof}
	Immediate from the construction of $S'$.
\end{proof}

\begin{corollary} \label{matrixS'}
	The matrix $S'$ is the transition matrix from 
	$$\{A_2E_i^*A-AE_i^*A_2\}_{i=1}^D, 
	\{A_3E_i^*-E_i^*A_3\}_{i=1}^D, 
	\{A_2E_i^*-E_i^*A_2\}_{i=1}^D, 
	\{AE_i^*-E_i^*A\}_{i=1}^D$$ 
	\noindent to 
	$$\{A_2A_i^*A-AA_i^*A_2\}_{i=1}^D, 
	\{A_3A_i^*-A_i^*A_3\}_{i=1}^D, 
	\{A_2A_i^*-A_i^*A_2\}_{i=1}^D, 
	\{AA_i^*-A_i^*A\}_{i=1}^D.$$ 
\end{corollary}
\begin{proof}
	Immediate from Lemma \ref{transS^alt}.
\end{proof}


\section{The bilinear form $\langle \;,\;\rangle $}

Recall the positive definite Hermitean bilinear form $\langle \;,\;\rangle $ on ${\rm{Mat}}_X(\mathbb{C})$.

\begin{lemma}\label{formula} $\rm{(See ~\cite[Lemma~ 3.2]{T_subconstituentI}.)}$ For $0\leq h, i, j, r, s, t \leq D$,
	\begin{enumerate}[(i)]
		\item $\langle E_i^*A_jE_h^*,E_r^*A_sE_t^*\rangle =\delta_{ir}\delta_{js}\delta_{ht}k_hp^h_{ij}$,
		\item $\langle E_iA_j^*E_h,E_rA_s^*E_t\rangle =\delta_{ir}\delta_{js}\delta_{ht}m_hq^h_{ij}$.
	\end{enumerate}
\end{lemma}

\begin{corollary} \label{p^h_ij,q^h_ij} $\rm{(See ~\cite[Lemma~ 3.2]{T_subconstituentI}.)}$
	For $0\leq h, i, j\leq D$,
	\begin{enumerate}[(i)]
		\item $E_i^*A_jE_h^*=0$ if and only if $p^h_{ij}=0$,
		\item $E_iA_j^*E_h=0$ if and only if $q^h_{ij}=0$.
	\end{enumerate}
\end{corollary}

\begin{lemma} \label{<>formula} 
	For $0\leq h, i, j, r, s, t \leq D$ we have
	\begin{equation*}
	\langle A_iE_j^*A_h,A_rE_s^*A_t\rangle =\displaystyle\sum_{\ell=0}^{D} k_\ell p^\ell_{ir}p^\ell_{js}p^\ell_{ht}.
	\end{equation*}
\end{lemma}	
\begin{proof}
	Since $A_iA_j=\sum_{h=0}^{D} p^h_{ij}A_h$ and $E_iE_j=\delta_{ij}E_i$	$(0\leq h, i, j \leq D)$ and by Lemma \ref{formula} and (\ref{k_hp^h_{ij}}), we obtain 
		\begin{align*} 
	\langle A_iE_j^*A_h,A_rE_s^*A_t\rangle  
	& = tr((A_iE_j^*A_h)^t(\overline{A_rE_s^*A_t})) \\
	& = tr(A_hE_j^*A_iA_rE_s^*A_t) 	\\
	& = \displaystyle\sum_{\ell=0}^{D} p^\ell_{ir}tr(A_hE_j^*A_\ell E_s^*A_t) \\
	& = \displaystyle\sum_{\ell=0}^{D} p^\ell_{ir}tr(E_j^*A_\ell E_s^*A_tA_h) \\
	& = \displaystyle\sum_{\ell=0}^{D}\displaystyle\sum_{w=0}^{D} p^\ell_{ir}p^w_{ht}tr(E_j^*A_\ell E_s^*A_w) \\
	& = \displaystyle\sum_{\ell=0}^{D}\displaystyle\sum_{w=0}^{D} p^\ell_{ir}p^w_{ht}tr(E_j^*E_j^*A_\ell E_s^*E_s^*A_w) \\
	& = \displaystyle\sum_{\ell=0}^{D}\displaystyle\sum_{w=0}^{D} p^\ell_{ir}p^w_{ht}tr(E_j^*A_\ell E_s^*E_s^*A_wE_j^*) \\
	& = \displaystyle\sum_{\ell=0}^{D}\displaystyle\sum_{w=0}^{D} p^\ell_{ir}p^w_{ht}tr((E_s^*A_\ell E_j^*)^t(\overline{E_s^*A_wE_j^*})) \\
	& = \displaystyle\sum_{\ell=0}^{D}\displaystyle\sum_{w=0}^{D} p^\ell_{ir}p^w_{ht}\langle E_s^*A_\ell E_j^*,E_s^*A_wE_j^*\rangle  \\
	& = \displaystyle\sum_{\ell=0}^{D}\displaystyle\sum_{w=0}^{D} \delta_{\ell w}p^\ell_{ir}p^w_{ht}p^j_{s\ell}k_j \\
	& = \displaystyle\sum_{\ell=0}^{D} k_\ell p^\ell_{ir}p^\ell_{js}p^\ell_{ht}. \qedhere
	\end{align*}
\end{proof}

\begin{definition} \label{G}
	Let $G$ denote the matrix of inner products for	
	$$A_2E_i^*A-AE_i^*A_2,A_3E_i^*-E_i^*A_3, A_2E_i^*-E_i^*A_2, AE_i^*-E_i^*A,$$
	where $1\leq i \leq D$. 
	Thus the matrix $G$ is $4D\times 4D$.
\end{definition}

\begin{theorem} \label{entries_G}
	The entries of $G$ are as follows:
	For $1\leq i,j \leq D$,
	\begin{table}[h!]
		\begin{adjustbox}{width=1\textwidth}
			\small
			\centering
			\begin{tabular}{ c|c c c c } 
				$\langle \;,\;\rangle $ & $A_2E_j^*A-AE_j^*A_2$ & $A_3E_j^*-E_j^*A_3$ & $A_2E_j^*-E_j^*A_2$ & $AE_j^*-E_j^*A$ \\ [.4em]
				\hline \\[-.8em]
				$A_2E_i^*A-AE_i^*A_2$ 
				& $\phi$ 
				& $2k_2b_2(p^1_{ij}-p^2_{ij})$ 
				& $2k_2a_2(p^1_{ij}-p^2_{ij})$ 
				& $2k_2c_2(p^1_{ij}-p^2_{ij})$ \\ [.4em]
				$A_3E_i^*-E_i^*A_3$ 
				& $2k_2b_2(p^1_{ij}-p^2_{ij})$ 
				& $2k_3(\delta_{ij}k_i-p^3_{ij})$ 
				& $0$ 
				& $0$ \\ [.4em]
				$A_2E_i^*-E_i^*A_2$ 
				& $2k_2a_2(p^1_{ij}-p^2_{ij})$ 
				& $0$ 
				& $2k_2(\delta_{ij}k_i-p^2_{ij})$ 
				& $0$ \\ [.4em]
				$AE_i^*-E_i^*A$ 
				& $2k_2c_2(p^1_{ij}-p^2_{ij})$ 
				& $0$ 
				& $0$ 
				& $2k(\delta_{ij}k_i-p^1_{ij})$ 
			\end{tabular}
		\end{adjustbox}
	\end{table} 

\noindent where $\phi/2$ is a weighted sum involving the following terms and coefficients:

\begin{table}[h!]
		\centering
		\begin{tabular}{ c|c } 
			term & coefficient  \\ [.4em]
			\hline \\[-.8em]
			$p^0_{ij}$ 
			& $kk_2$  \\ [.4em]
			$p^1_{ij}$ 
			& $k_2a_1a_2-kb_1^2$ \\ [.4em]
			$p^2_{ij}$ 
			& $k_2(c_2(b_1-1)-a_2(a_1+1) +b_2(c_3-1))$ \\ [.4em]
			$p^3_{ij}$
			&  $-k_3c_3^2$ 
		\end{tabular}
\end{table} 
%
	\end{theorem}
\begin{proof} By Lemma \ref{<>formula} and using (\ref{p^h_{0j}})--(\ref{formula ai,bi,ci}), we obtain 
	\begin{align} \label{A_2E_i^*A-AE_i^*A_2,A_2E_j^*A-AE_j^*A_2}
	&\langle A_2E_i^*A-AE_i^*A_2,A_2E_j^*A-AE_j^*A_2\rangle   \nonumber\\
	&\quad =\langle A_2E_i^*A,A_2E_j^*A\rangle -\langle A_2E_i^*A,AE_j^*A_2\rangle  -\langle AE_i^*A_2,A_2E_j^*A\rangle +\langle AE_i^*A_2,AE_j^*A_2\rangle  \nonumber\\
	&\quad =\displaystyle\sum_{\alpha=0}^{D} k_{\alpha}p^{\alpha}_{22}p^{\alpha}_{ij}p^{\alpha}_{11}
	-\displaystyle\sum_{\beta=0}^{D} k_{\beta}p^{\beta}_{21}p^{\beta}_{ij}p^{\beta}_{12} 
	-\displaystyle\sum_{\gamma=0}^{D} k_{\gamma}p^{\gamma}_{12}p^{\gamma}_{ij}p^{\gamma}_{21}
	+\displaystyle\sum_{\eta=0}^{D} k_{\eta}p^{\eta}_{11}p^{\eta}_{ij}p^{\eta}_{22}  \nonumber\\
	&\quad = 2\bigg(\displaystyle\sum_{\alpha=0}^{2} k_{\alpha}p^{\alpha}_{22}p^{\alpha}_{ij}p^{\alpha}_{11} 
	-\displaystyle\sum_{\beta=1}^{3} k_{\beta}(p^{\beta}_{12})^2p^{\beta}_{ij}\bigg) \nonumber\\
	&\quad = 2(k_0p^0_{22}p^0_{ij}p^0_{11} + k_1p^1_{22}p^1_{ij}p^1_{11} + k_2p^2_{22}p^2_{ij}p^2_{11}
	-k_1(p^1_{12})^2p^1_{ij}-k_2(p^2_{12})^2p^2_{ij}-k_3(p^3_{12})^2p^3_{ij}) \nonumber\\
	&\quad =2(kk_2p^0_{ij}+(k_2a_1a_2-kb_1^2)p^1_{ij} +k_2(c_2(b_1-1)-a_2(a_1+1)+b_2(c_3-1))p^2_{ij} \nonumber\\ 
	&\qquad -k_3c_3^2p^3_{ij}).
	\end{align}

	\noindent Similarly, for $1 \leq h \leq 3$,	
	\begin{align} \label{A_hE_i^*-E_i^*A_h,A_2E_j^*A-AE_j^*A_2}
	&\langle A_hE_i^*-E_i^*A_h,A_2E_j^*A-AE_j^*A_2\rangle \nonumber\\
	&\quad =\langle A_hE_i^*,A_2E_j^*A\rangle -\langle A_hE_i^*,AE_j^*A_2\rangle  
	 -\langle E_i^*A_h,A_2E_j^*A\rangle +\langle E_i^*A_h,AE_j^*A_2\rangle  \nonumber\\
	&\quad =\langle A_hE_i^*A_0,A_2E_j^*A\rangle -\langle A_hE_i^*A_0,AE_j^*A_2\rangle  
	-\langle A_0E_i^*A_h,A_2E_j^*A\rangle  \nonumber\\
	&\qquad +\langle A_0E_i^*A_h,AE_j^*A_2\rangle  \nonumber\\
	&\quad =\displaystyle\sum_{\alpha=0}^{D} k_{\alpha}p^{\alpha}_{h2}p^{\alpha}_{ij}p^{\alpha}_{01}
	-\displaystyle\sum_{\beta=0}^{D} k_{\beta}p^{\beta}_{h1}p^{\beta}_{ij}p^{\beta}_{02}
	-\displaystyle\sum_{\gamma=0}^{D} k_{\gamma}p^{\gamma}_{02}p^{\gamma}_{ij}p^{\gamma}_{h1}
	+\displaystyle\sum_{\eta=0}^{D} k_{\eta}p^{\eta}_{01}p^{\eta}_{ij}p^{\eta}_{h2}  \nonumber\\
	&\quad = 2(kp^1_{h2}p^1_{ij}-k_2p^2_{h1}p^2_{ij})  \nonumber\\
	&\quad = 2(k_2p^2_{1h}p^1_{ij}-k_2p^2_{1h}p^2_{ij})  \nonumber\\
	&\quad =2k_2p^2_{1h}(p^1_{ij}-p^2_{ij}). 
	\end{align}

	\noindent Similarly, for $1 \leq h,\ell\leq 3$,		
	\begin{align}  \label{A_hE_i^*-E_i^*A_h,A_lE_j^*-E_j^*A_l}
	&\langle A_hE_i^*-E_i^*A_h,A_\ell E_j^*-E_j^*A_\ell\rangle  \nonumber\\
	&\qquad=\langle A_hE_i^*,A_\ell E_j^*\rangle -\langle A_hE_i^*,E_j^*A_\ell\rangle 
	-\langle E_i^*A_h,A_\ell E_j^*\rangle  
	 +\langle E_i^*A_h,E_j^*A_\ell\rangle  \nonumber\\
	&\qquad =\langle A_hE_i^*A_0,A_\ell E_j^*A_0\rangle -\langle A_hE_i^*A_0,A_0E_j^*A_\ell\rangle  
	-\langle A_0E_i^*A_h,A_\ell E_j^*A_0\rangle  \nonumber\\
	&\quad \qquad +\langle A_0E_i^*A_h,A_0E_j^*A_\ell\rangle  \nonumber\\
	&\qquad =\displaystyle\sum_{\alpha=0}^{D} k_{\alpha}p^{\alpha}_{h\ell}p^{\alpha}_{ij}p^{\alpha}_{00}
	-\displaystyle\sum_{\beta=0}^{D} k_{\beta}p^{\beta}_{h0}p^{\beta}_{ij}p^{\beta}_{0\ell}
	-\displaystyle\sum_{\gamma=0}^{D} k_{\gamma}p^{\gamma}_{0\ell}p^{\gamma}_{ij}p^{\gamma}_{h0} 
	+\displaystyle\sum_{\eta=0}^{D} k_{\eta}p^{\eta}_{00}p^{\eta}_{ij}p^{\eta}_{h\ell}  \nonumber\\	
	&\qquad =2(k_0p^0_{h\ell}p^0_{ij}-\delta_{h\ell}k_hp^h_{ij})  \nonumber\\
	&\qquad =2(\delta_{h\ell}\delta_{ij}k_hk_i-\delta_{h\ell}k_hp^h_{ij}) \nonumber \\
	&\qquad =2\delta_{h\ell}k_h(\delta_{ij}k_i-p^h_{ij}).
	\end{align}
The result follows.  
\end{proof}

In Appendix 2, we give the matrix $G$ for $D=3$.

\begin{definition} \label{B_i}
	For $1\leq i \leq D$ let $B_i$ denote the matrix of inner products for 
	$$A_2A_i^*A-AA_i^*A_2,A_3A_i^*-A_i^*A_3, A_2A_i^*-A_i^*A_2, AA_i^*-A_i^*A.$$
	So the matrix $B_i$ is $4\times 4$.
\end{definition}

\begin{definition} \label{widetildeG}
	Let $\widetilde{G}$ denote the matrix of inner products for	
	$$A_2A_i^*A-AA_i^*A_2,A_3A_i^*-A_i^*A_3, A_2A_i^*-A_i^*A_2, AA_i^*-A_i^*A,$$
	where $1\leq i \leq D$.
	Thus the matrix $\widetilde{G}$ is $4D\times 4D$.
\end{definition}

\begin{lemma} \label{form_tildeG}
	The matrix $\widetilde{G}$ has the form
	\[
	\widetilde{G}=\begin{bmatrix}
	B_1    &        &        & \bf{0}  \\
	& B_2    &        &   \\
	&        & \ddots &   \\
	\bf{0}      &        &        & B_D 
	\end{bmatrix},
	\]
where $B_1, B_2, ..., B_D$ are from Definition $\rm{\ref{B_i}}$.
\end{lemma}
\begin{proof}
	Recall that $A_0^*,A_1^*,...,A_D^*$ form a basis for $M^*$.
	Therefore the sum $MM^*M=\sum_{i=0}^{D}MA_i^*M$ is direct.
	The summands are mutually orthogonal by Lemma \ref{formula}(ii).
	The result follows.
\end{proof}

\begin{lemma} \label{det_tildeG}
	$det(\widetilde{G})=\displaystyle\prod_{i=1}^{D} det(B_i)$.
\end{lemma}
\begin{proof}
	Immediate from Lemma \ref{form_tildeG}.
\end{proof}


\section{The main result}

In this section we obtain our main result, which is Theorem \ref{primitive}.

\begin{lemma} \label{det}
	The following (i)--(iii) hold.
	\begin{enumerate}[(i)]
		\item $\widetilde{G}=(S')^tGS'$.
		\item $det(G)=\big(det(S')\big)^{-2}det(\widetilde{G})$.
		\item $det(G)=\big(det(S^{alt})\big)^{-8}\displaystyle\prod_{i=1}^{D} det(B_i)$.
	\end{enumerate}
\end{lemma}
\begin{proof}
	$(i)$ Immediate from Definition \ref{G}, Definition \ref{widetildeG}, and Corollary \ref{matrixS'}.
	
	\noindent $(ii)$ Follows from $(i)$.
	
	\noindent $(iii)$ Follows from $(ii)$ and Lemmas \ref{det_S'}, \ref{det_tildeG}.
\end{proof}

\begin{theorem} \label{primitive} Let $\Gamma$ denote a primitive distance-regular graph with diameter $D\geq3$. 
Then  $\Gamma$ is $Q$-polynomial if and only if $det(G)=0$.
\end{theorem}	
\begin{proof}
	To prove the theorem in one direction, assume that $\Gamma$ is $Q$-polynomial with respect to the ordering $E_0,E_1,...,E_D$.
	Write $A^*=A^*_1$.
	By \cite[Theorem 3.3]{T_New} and Lemma \ref{primitive_S_alt}, we obtain $A^*A_3-A_3A^*\in{Span\{AA^*A_2-A_2A^*A,A^*A_2-A_2A^*,A^*A-AA^*\}}$.
	Thus $AA^*A_2-A_2A^*A,A^*A_3-A_3A^*,A^*A_2-A_2A^*,A^*A-AA^*$ are linearly dependent.
	Now the matrix $B_1$ from  Definition \ref{B_i} has determinant zero.
	Now $det(G)=0$ by Lemma \ref{det}(iii).

%
%

	For the other direction, assume $det(G)=0$.
	By Lemma \ref{det}(iii) and since $S^{alt}$ is invertible, there exists an integer $t$ $(1\leq t\leq D)$ such that $det(B_t)=0$.
	Now $AA_t^*A_2-A_2A_t^*A,A_t^*A_3-A_3A_t^*,A_t^*A_2-A_2A_t^*,A_t^*A-AA_t^*$ are linearly dependent.
	We will show that $A_t^*A_3-A_3A_t^*\in{Span\{AA_t^*A_2-A_2A_t^*A,A_t^*A_2-A_2A_t^*,A_t^*A-AA_t^*\}}$.
	To do this we show that $AA_t^*A_2-A_2A_t^*A,A_t^*A_2-A_2A_t^*,A_t^*A-AA_t^*$  are linearly independent.
	Suppose not.
	Then there exist scalars $a,b,c$, not all zero, such that
	\begin{align}\label{lin dep}
	a(AA_t^*A_2-A_2A_t^*A)+b(A_t^*A_2-A_2A_t^*)+c(A_t^*A-AA_t^*)=0.
	\end{align}
	Abbreviate $\theta^*_i=m_tu_i(\theta_t)$ $(0\leq i \leq D)$.
	So $A^*_t=\sum_{i=0}^{D}\theta^*_iE^*_i$.
	By Lemma \ref{primitive_S_alt}, 
	\begin{align}\label{lem1}
	\qquad \qquad \qquad \qquad \qquad \theta^*_i\neq \theta^*_0 
	\qquad \qquad \qquad \qquad \qquad (1\leq i \leq D).
	\end{align}
	For $1\leq h \leq 3$ pick $z\in X$ such that $\partial(x,z)=h$.
	Compute the $(x,z)$-entry in (\ref{lin dep}).
	For $h=3$ we get $ac_3(\theta^*_1-\theta^*_2)=0$.
	For $h=2$ we get  $aa_2(\theta^*_1-\theta^*_2)+b(\theta^*_0-\theta^*_2)=0$.
	For $h=1$ we get $ab_1(\theta^*_1-\theta^*_2)+c(\theta^*_0-\theta^*_1)=0$.
	Solving these equations we obtain $a(\theta^*_1-\theta^*_2)=0$ and  $b=0$, $c=0$.
	Recall that $a, b, c$ are not all zero, so $a\neq 0$ and $\theta^*_1=\theta^*_2$.
	Now (\ref{lin dep}) becomes $AA_t^*A_2-A_2A_t^*A=0$.
	Recall $c_2A_2=A^2-a_1A-kI$. 
	We have $AA^*_tA^2+kA^*_tA=A^2A^*_tA+kAA^*_t$.
	Thus $[A,AA^*_tA+kA^*_t]=0$.
	For $0\leq i, j\leq D$ such that $i\neq j$ we have $E_iA^*_tE_j(\theta_i\theta_j+k)=0$.
	By Corollary \ref{p^h_ij,q^h_ij}, $E_iA^*_tE_j\neq 0$ if and only if $q^t_{ij}\neq 0$, and in this case $\theta_i\theta_j+k=0$.
	Since $q^t_{0t}=1$ and $\theta_0=k$, we have $k\theta_t+k=0$ and hence $\theta_t=-1$.
	Define a diagram with nodes $0,1,...,D$.
	There exists an arc between nodes $i,j$ if and only if $i\neq j$ and $q^t_{ij}\neq 0$.
	Observe that node $0$ is connected to node $t$ and no other nodes.
	By \cite [Proposition 2.11.1]{BCN} and Lemma \ref{primitive_S_alt}, the diagram is connected.
	Thus there exists a node $s$ with $s\neq 0$ and $s\neq t$ that is connected to node $t$ by an arc.
	In other words $q^t_{st}\neq 0$.
	So $\theta_s\theta_t+k=0$ and hence $\theta_s=k$, a contradiction.
	Therefore $AA_t^*A_2-A_2A_t^*A,A_t^*A_2-A_2A_t^*,A_t^*A-AA_t^*$  are linearly independent.
	So $A_t^*A_3-A_3A_t^*\in{Span\{AA_t^*A_2-A_2A_t^*A,A_t^*A_2-A_2A_t^*,A_t^*A-AA_t^*\}}$.
	Now by \cite[Theorem 3.3]{T_New} and (\ref{lem1}), $\Gamma$ is a $Q$-polynomial with respect to $E=E_t$.
\end{proof}


\section{Appendix 1}

Recall the distance-regular graph $\Gamma$ with diameter $D$.
Recall for $0\leq h \leq D$ 
\begin{table}[h!]
	\centering
	\begin{tabular}{ l l l l l } 
		$p^h_{1,h-1}=c_h$,
		&
		& $p^h_{1h}=a_h$,  
		&
		&$p^h_{1,h+1}=b_h$,  \\ [.4em]
		$p^1_{h,h-1}=\dfrac{k_hc_h}{k}$, 
		&
		& $p^1_{hh}=\dfrac{k_ha_h}{k}$,
		&
		&$p^1_{h,h+1}=\dfrac{k_hb_h}{k}$. 
	\end{tabular}
\end{table} 

\noindent We now give $p^h_{2j}$ for $h-2\leq j \leq h+2$.
\begin{flalign*}
\qquad \quad &p^h_{2,h-2}=\dfrac{c_{h-1}c_h}{c_2}, &&\\
&p^h_{2,h-1}=\dfrac{c_h(a_{h-1}+a_h-a_1)}{c_2}, \\
&p^h_{2h}= \dfrac{c_h(b_{h-1}-1)+a_h(a_h-a_1-1)+b_h(c_{h+1}-1)}{c_2}, \\
&p^h_{2,h+1}= \dfrac{b_h(a_{h+1}+a_h-a_1)}{c_2},\\
&p^h_{2,h+2}= \dfrac{b_hb_{h+1}}{c_2}.
\end{flalign*}

\noindent We now give $p^h_{3j}$ for $h-3\leq j \leq h+3$.
\begin{flalign*}
\qquad \quad
&p^h_{3,h-3}= \dfrac{c_{h-2}c_{h-1}c_h}{c_2c_3},&&\\
&p^h_{3,h-2}=\dfrac{(a_h-a_2)c_{h-1}c_h}{c_2c_3}+\dfrac{c_{h-1}c_h(a_{h-2}+a_{h-1}-a_1)}{c_2c_3},\\
&p^h_{3,h-1}=\dfrac{c_{h-1}c_h(b_{h-2}-1)+c_ha_{h-1}(a_{h-1}-a_1-1)+c_hb_{h-1}(c_h-1)}{c_2c_3}\\
& \qquad \qquad +\dfrac{c_h(a_h-a_2)(a_{h-1}+a_h-a_1)}{c_2c_3}+\dfrac{b_hc_hc_{h+1}}{c_2c_3}-\dfrac{b_1c_h}{c_3}, 
\end{flalign*}
\begin{flalign*}
\qquad \quad
&p^h_{3h}= \dfrac{c_hb_{h-1}(a_{h}+a_{h-1}-a_1)}{c_2c_3} &&\\
& \qquad \quad +\dfrac{(a_h-a_2)(c_h(b_{h-1}-1)+a_h(a_h-a_1-1)+b_h(c_{h+1}-1))}{c_2c_3}\\
& \qquad \quad +\dfrac{b_hc_{h+1}(a_h+a_{h+1}-a_1)}{c_2c_3}-\dfrac{b_1a_h}{c_3}, \\
&p^h_{3,h+1}= \dfrac{c_hb_{h-1}b_h}{c_2c_3} 
+\dfrac{b_h(a_h-a_2)(a_{h+1}+a_h-a_1)}{c_2c_3}\\
& \qquad \qquad  +\dfrac{b_h(c_{h+1}(b_h-1)+a_{h+1}(a_{h+1}-a_1-1)+b_{h+1}(c_{h+2}-1))}{c_2c_3}-\dfrac{b_1b_h}{c_3}, \\
&p^h_{3,h+2}= \dfrac{(a_h-a_2)b_hb_{h+1}}{c_2c_3} +\dfrac{b_hb_{h+1}(a_{h+2}+a_{h+1}-a_1)}{c_2c_3}, \\
&p^h_{3,h+3}= \dfrac{b_hb_{h+1}b_{h+2}}{c_2c_3}.
\end{flalign*}

\section{Appendix 2}

Recall the matrix $G$ from Theorem \ref{entries_G}.
In this appendix we  give $G$ for $D=3$.

\begin{example}
	Assume $D=3$.
	The rows and columns of $G$ are indexed by the following matrices, in the specified order:
\begin{center}
\begin{tabular}{cclllll}
	block 1 & & $A_3E_1^*-E_1^*A_3$ & & $A_3E_2^*-E_2^*A_3$ & & $A_3E_3^*-E_3^*A_3$ \\	
	block 2 & & $A_2E_1^*-E_1^*A_2$ & & $A_2E_2^*-E_2^*A_2$ & & $A_2E_3^*-E_3^*A_2$ \\	
	block 3 & & $AE_1^*-E_1^*A$ 	  & & $AE_2^*-E_2^*A$ 	& & $AE_3^*-E_3^*A$ \\	
	block 4 & & $A_2E_1^*A-AE_1^*A_2$ & & $A_2E_2^*A-AE_2^*A_2$ & & $A_2E_3^*A-AE_3^*A_2$.\\
\end{tabular}
\end{center}
\noindent So the matrix $G$ is $12\times 12$.
$G$ has the form	
	\[
	G=\begin{bmatrix}
	\mathbb{X}  	& 0	   		 	&  0      		& \mathbb{S}  \\
	0        		& \mathbb{Y}   	&  0      		& \mathbb{T}  \\
	0        		& 0        		&  \mathbb{Z} 	& \mathbb{U}  \\
	\mathbb{S}   	& \mathbb{T}	&  \mathbb{U}	& \mathbb{W}  \\
	\end{bmatrix},
	\]
where each block is a $3\times 3$ symmetric matrix as shown below.

\[\mathbb{X}=
\begin{bmatrix}
2k_3k					& -2k_3c_3   			&  -2k_3a_3  \\
-2k_3c_3	     		& 2k_3(k_2-p^3_{22})   	&  -2k_3p^3_{23}  \\
-2k_3a_3		  		& -2k_3p^3_{23} 		& 2k_3(k_3-p^3_{33}) 	  
\end{bmatrix},
\]

\[\mathbb{Y}=
\begin{bmatrix}
2k_2(k-c_2)				& -2k_2a_2		   		&  -2k_2b_2  \\
-2k_2a_2    			& 2k_2(k_2-p^2_{22})   	&  -2k_2p^2_{23}  \\
-2k_2b_2		  		& -2k_2p^2_{23} 		& 2k_2(k_3-p^2_{33}) 	  
\end{bmatrix},
\]

\[\mathbb{Z}=
\begin{bmatrix}
2k(k-a_1)	 			& -2kb_1		   		&  0  \\
-2kb_1		    		& 2k(k_2-p^1_{22})   	&  -2kp^1_{23}  \\
0			   			& -2kp^1_{23} 			& 2k(k_3-p^1_{33}) 	  
\end{bmatrix},
\]

\[\mathbb{S}=
\begin{bmatrix}
2k_2b_2(a_1-c_2) 	& 2k_2b_2(b_1-a_2)   &  -2k_2b_2^2   \\
2k_2b_2(b_1-a_2)    & 2k_2b_2(p^1_{22}-p^2_{22})   	
&  2k_2b_2(p^1_{23}-p^2_{23})  \\
-2k_2b_2^2			& 2k_2b_2(p^1_{23}-p^2_{23}) 	
	& 2k_2b_2(p^1_{33}-p^2_{33})	  
\end{bmatrix},
\]

\[\mathbb{T}=
\begin{bmatrix}
2k_2a_2(a_1-c_2) 		& 2k_2a_2(b_1-a_2)   &  -2k_2a_2b_2   \\
2k_2a_2(b_1-a_2)     	& 2k_2a_2(p^1_{22}-p^2_{22})   	
&  2k_2a_2(p^1_{23}-p^2_{23})  \\
-2k_2a_2b_2  	& 2k_2a_2(p^1_{23}-p^2_{23}) 	
& 2k_2a_2(p^1_{33}-p^2_{33})	  
\end{bmatrix},
\]

\[\mathbb{U}=
\begin{bmatrix}
2k_2c_2(a_1-c_2) 		& 2k_2c_2(b_1-a_2)   &  -2k_2c_2b_2   \\
2k_2c_2(b_1-a_2)     	& 2k_2c_2(p^1_{22}-p^2_{22})   	
&  2k_2c_2(p^1_{23}-p^2_{23})  \\
-2k_2c_2b_2  & 2k_2c_2(p^1_{23}-p^2_{23}) &2k_2c_2(p^1_{33}-p^2_{33})	  
\end{bmatrix}.
\]

\noindent The matrix $\mathbb{W}$ is symmetric with entries
\begin{align*}
{\mathbb W}_{11}
&=2\big(k^2k_2+\big(k_2a_1a_2-kb_1^2\big)a_1 			+\big(k_2(c_2(b_1-1)+a_2(a_2-a_1-1)&&\\
&\quad +b_2(c_3-1))-k_2a_2^2\big)c_2\big),\\
{\mathbb W}_{12}
&=2\big(\big(k_2a_1a_2-kb_1^2\big)b_1 +\big(k_2(c_2(b_1-1)+a_2(a_2-a_1-1)+b_2(c_3-1)) &&\\
&\quad -k_2a_2^2\big)a_2-k_3c_3^3\big),\\
{\mathbb W}_{13}
&=2\big(\big(k_2(c_2(b_1-1)+a_2(a_2-a_1-1)+b_2(c_3-1))-k_2a_2^2\big)b_2-k_3c_3^2a_3\big),\\
{\mathbb W}_{22}
&=2\big(kk_2^2+\big(k_2a_1a_2-kb_1^2\big)p^1_{22} +\big(k_2(c_2(b_1-1)+a_2(a_2-a_1-1) &&\\
&\quad +b_2(c_3-1))-k_2a_2^2\big)p^2_{22}-k_3c_3^2p^3_{22}\big),\\
{\mathbb W}_{23}
&=2\big(\big(k_2a_1a_2-kb_1^2\big)p^1_{23} +\big(k_2(c_2(b_1-1)+a_2(a_2-a_1-1)+b_2(c_3-1)) &&\\
&\quad -k_2a_2^2\big)p^2_{23}-k_3c_3^2p^3_{23}\big),\\
{\mathbb W}_{33}
&=2\big(kk_2k_3+\big(k_2a_1a_2-kb_1^2\big)p^1_{33} +\big(k_2(c_2(b_1-1)+a_2(a_2-a_1-1)&&\\
&\quad +b_2(c_3-1))-k_2a_2^2\big)p^2_{33}-k_3c_3^2p^3_{33}\big).
\end{align*}

\noindent From Appendix 1, we find
\begin{flalign*}
&p^1_{22}=\dfrac{k_2a_2}{k}, \enskip p^2_{22}=\dfrac{c_2(b_1-1)+a_2(a_2-a_1-1)+b_2(c_3-1)}{c_2}, 
\enskip p^3_{22}=\dfrac{c_3(a_2+a_3-a_1)}{c_2}, \\
&p^1_{23}=\dfrac{k_2b_2}{k},	\enskip p^2_{23}=\dfrac{b_2(a_3+a_2-a_1)}{c_2}, \enskip p^3_{23}=\dfrac{c_3(b_2-1)+a_3(a_3-a_1-1)-b_3}{c_2}, \\
&p^1_{33}=\dfrac{k_3a_3}{k},	
\enskip p^2_{33}=\dfrac{b_2(c_3(b_2-1)+a_3(a_3-a_1-1)-b_3)}{c_2c_3}, &&\\
&p^3_{33}=\dfrac{c_3b_{2}(a_3+a_2-a_1)}{c_2c_3} +\dfrac{(a_3-a_2)(c_3(b_2-1)+a_3(a_3-a_1-1)-b_3)}{c_2c_3} -\dfrac{b_1a_3}{c_3}. 
\end{flalign*}
\end{example}


\section{Acknowledgement}

The author would like to thank Professor Paul Terwilliger for many valuable ideas and suggestions.
This paper was written while the author was an Honorary Fellow at the University of Wisconsin-Madison (January 2017- January 2018) supported by  the Development and Promotion of Science and Technology Talents (DPST)
Project, Thailand.


%

\end{document}